\documentclass[12pt]{article}
\usepackage{amsthm,amsmath,amssymb,extsizes,color,graphicx,hyperref,subcaption}
\usepackage[active]{srcltx}
\usepackage{MnSymbol}
\usepackage{tikz}

\title{
Recovering a perturbation of a matrix polynomial from a perturbation of its linearization 
}
\author{Andrii Dmytryshyn$^*$}

\date{\vskip-0.4cm\small{$^*$School of Science and Technology, \"Orebro University, 701 82 \"Orebro, Sweden}
\newline 
\vskip0.28cm
\small{\rm Dedicated to the memory of my father, Roman Dmytryshyn (1959--2020)}
\vskip-0.2cm}

\DeclareMathOperator{\orb}{O}

\DeclareMathOperator{\inv}{ i}
\DeclareMathOperator{\co}{c}
\DeclareMathOperator{\GSYL}{GSYL}

\DeclareMathOperator{\im}{Im}
\DeclareMathOperator{\re}{Re}
\DeclareMathOperator{\vv}{vec}
\DeclareMathOperator{\Ind}{Ind}
\DeclareMathOperator{\grade}{grade}

\newcommand{\ddd}{
\text{\begin{picture}(12,8)
\put(-2,-4){$\cdot$}
\put(3,0){$\cdot$}
\put(8,4){$\cdot$}
\end{picture}}}

\renewcommand{\le}{\leqslant}
\renewcommand{\ge}{\geqslant}

\newtheorem{theorem}{Theorem}[section]
\newtheorem{lemma}{Lemma}[section]
\newtheorem{algorithm}{Algorithm}[section]

\theoremstyle{definition}
\newtheorem{definition}{Definition}[section]

\theoremstyle{remark}
\newtheorem{remark}{Remark}[section]
\newtheorem{corollary}{Corollary}[section]
\newtheorem{example}{Example}[section]

\newcommand{\hide}[1]{}

\usepackage{listings}
\usepackage{color} 
\definecolor{mygreen}{RGB}{28,172,0} 
\definecolor{mylilas}{RGB}{170,55,241}

\begin{document}
\maketitle

\begin{abstract}


A number of theoretical and computational problems for matrix polynomials are solved by passing to linearizations. Therefore a perturbation theory results for linearizations need to be related back to matrix polynomials. In this paper we present an algorithm that finds which perturbation of matrix coefficients of a matrix polynomial corresponds to a given perturbation of the entire linearization pencil. Moreover we find transformation matrices that, via strict equivalence, transform a perturbation of the linearization to the linearization of a perturbed polynomial. For simplicity, we present the results for the first companion linearization but they can be generalized to a broader class of linearizations.  
%
\end{abstract}



\section{Introduction} 
Nonlinear eigenvalue problems play an important role in mathematics and its applications, see e.g., the surveys \cite{GuTi17,MaMT15,TiMe01}. 
In particular, polynomial eigenvalue problems have been receiving much attention \cite{BHMS13,DJKV19,DLPV18,HiMM04,KaTi15,KrSW09}. 
Recall that 
\begin{equation} 
\label{matpol}
P(\lambda) = \lambda^{d}A_{d} + \dots +  \lambda A_1 + A_0, 
\quad \ A_i \in \mathbb C^{m \times n},  \text{  and } i=0, \dots, d 
\end{equation}
is a {\it matrix polynomial} and that the number $d$ is called a {\it grade} of $P(\lambda)$. If $A_d \neq 0$ then the grade coincides with the {\it degree} of a polynomial. Frequently, complete eigenstructures, i.e. elementary divisors and minimal indices of matrix polynomials (for the definitions, see e.g., \cite{DeDM12,DJKV19}) provide an understanding of properties and behaviours of the underlying physical systems and thus are the actual objects of interest. 
Complete eigenstructure is usually computed by passing to a {\it (strong) linearization} which replaces a matrix polynomial by a matrix pencil, i.e. matrix polynomials of degree $d = 1$, with the same finite (and infinite) elementary divisors and with the known changes in the minimal indices, see more details in \cite{MaMT15}. For example, a classical linearization of \eqref{matpol}, used in this paper, is the first companion form 
\begin{equation}
\label{1stform}
{\cal C}^1_{P(\lambda)}=\lambda \begin{bmatrix}
A_d&&&\\
&I_n&&\\
&&\ddots&\\
&&&I_n\\
\end{bmatrix} + \begin{bmatrix}
A_{d-1}&A_{d-2}&\dots&A_{0}\\
-I_n&&&\\ 
&\ddots&&\\
&&-I_n&\\
\end{bmatrix}, 
\end{equation}
where $I_n$ is the $n\times n$ identity matrix and all nonspecified blocks are zeros. 

In this paper, we find which perturbation of matrix coefficients of a given matrix polynomial corresponds to a given perturbation of the entire linearization pencil. 
To be exact, we find such a perturbation of matrix polynomial coefficients that the linearization of this perturbed polynomial \eqref{deformlin}, has the same complete eigenstructure as a given perturbed linearization \eqref{fulldeform}. We also note that, the existence of such a perturbation \eqref{deformlin} was proven before for Fiedler-type linearizations \cite{Dmyt17,DJKV19,VaDe83}, and even for a larger class of block-Kronecker linearizations \cite{DLPV18}, but this existence also follows from the convergence of the algorithm developed in this paper. 
%



The results of this paper can be applied to a number of problems in numerical linear algebra. One example is solving various distance problems for matrix polynomials if the corresponding problems are solved for matrix pencils, e.g., finding a singular matrix polynomials nearby a given matrix polynomial \cite{ByHM98,GiHL17,GuLM17}. Another application lies in the stratification theory \cite{Dmyt17,DJKV19}: constructing an explicit perturbation of a matrix polynomial when a perturbation of its linearization is known. This will allow to say which perturbation does a certain change to the complete eigenstructure of a given polynomial. (In \cite{DFKK15,FuKS14,FKKS19} the explicit perturbations for investigating such changes for matrix pencils, bi- and sesquilinear forms are derived.) 
Moreover, our result may also be useful for investigating the backward stability of the polynomial eigenvalue problems sovled by using the backward stable methods on the linearizations, see e.g., \cite{Tiss00}.

\section{Perturbations of matrix polynomials and their linearizations} 
\label{sec:perturb}

Recall that for every matrix $X = [x_{ij}]$ its  Frobenius norm is given by $|| X || := || X ||_F= \left( \sum_{i,j} |x_{ij}|^2 \right)^{\frac{1}{2}}$. Hereafter, unless the otherwise is stated, we use the Frobenius norm for matrices. Let $P(\lambda)$ be an $m \times n$ matrix polynomial of grade $d$. Define a norm of a matrix polynomial $P(\lambda)$ as follows 
$$
|| P(\lambda) || := \left( \sum_{k=0}^d || A_k ||^2 \right)^{\frac{1}{2}}.
$$

\begin{definition} \label{perturb}
\ Let \ $P(\lambda)$ \ and \ $E(\lambda)$ \ be two $m \times n$ matrix polynomials, with $\grade P(\lambda) \geq \grade E(\lambda)$. A matrix polynomial $P(\lambda) + E(\lambda)$ is a {\it perturbation of an $m \times n$ matrix polynomial} $P(\lambda)$. 
\end{definition}
\noindent In this paper $|| E(\lambda) ||$ is typically small. Definition \ref{perturb} is also applicable to matrix pencils as a particular case of matrix polynomials.  

%
The first companion form ${\cal C}^1_{P(\lambda)}$ of $P(\lambda)$ is defined in \eqref{1stform} and is a well-known way to linearize matrix polynomials, i.e. to substitute an investigation of a matrix polynomial by an investigation of a certain matrix pencil with the same characteristics of interest. 
Namely, $P(\lambda)$ and ${\cal C}_{P(\lambda)}^1$ have the same elementary divisors (the same eigenvalues and their multiplicities), the same left minimal indices, and there is a simple relation between their right minimal indices (those of ${\cal C}_{P(\lambda)}^1$ are greater by $d-1$ than those of $P(\lambda)$), see \cite{DeDM12} for the definitions and more details. 
%
%
%
Define a (full) perturbation of the linearization of an $m \times n$ matrix polynomial of grade $d$ as follows 
\begin{equation}
\label{fulldeform}
\begin{aligned}
{\cal C}^1_{P(\lambda)} + {\cal E}&:= 
\lambda \begin{bmatrix}
A_d&&&\\
&I_n&&\\
&&\ddots&\\
&&&I_n\\
\end{bmatrix} + \begin{bmatrix}
A_{d-1}&A_{d-2}&\dots&A_{0}\\
-I_n&0&\dots&0\\
&\ddots&\ddots&\vdots\\
0&&-I_n&0\\
\end{bmatrix}\\
&+ \lambda  \begin{bmatrix}
E_{11}&E_{12}&E_{13}&\dots&E_{1d}\\
E_{21}&E_{22}&E_{23}&\dots&E_{2d}\\
E_{31}&E_{32}&E_{33}&\dots&E_{3d}\\
\vdots&\vdots&\vdots&\ddots&\vdots\\
E_{d1}&E_{d2}&E_{d3}&\dots&E_{dd}\\
\end{bmatrix} +  
\begin{bmatrix}
E'_{11}&E'_{12}&E'_{13}&\dots&E'_{1d}\\
E'_{21}&E'_{22}&E'_{23}&\dots&E'_{2d}\\
E'_{31}&E'_{32}&E'_{33}&\dots&E'_{3d}\\
\vdots&\vdots&\vdots&\ddots&\vdots\\
E'_{d1}&E'_{d2}&E'_{d3}&\dots&E'_{dd}\\
\end{bmatrix},
\end{aligned}
\end{equation}
and define a structured perturbation of the linearization, i.e. a perturbation in which only the blocks $A_i, i=0,1, \dots$ are perturbed
\begin{equation}
\begin{aligned}
\label{deformlin}
{\cal C}^1_{P(\lambda) + E(\lambda)}&= 
\lambda \begin{bmatrix}
A_d&&&\\
&I_n&&\\
&&\ddots&\\
&&&I_n\\
\end{bmatrix} + \begin{bmatrix}
A_{d-1}&A_{d-2}&\dots&A_{0}\\
-I_n&0&\dots&0\\
&\ddots&\ddots&\vdots\\
0&&-I_n&0\\
\end{bmatrix}\\
& + 
\lambda \begin{bmatrix}
F_d&0&\dots&0\\
0&0&\dots&0\\
\vdots&\vdots&\ddots&\vdots\\
0&0&\dots&0\\
\end{bmatrix} + \begin{bmatrix}
F_{d-1}&F_{d-2}&\dots&F_{0}\\
0&0&\dots&0\\
\vdots&\vdots&&\vdots\\
0&0&\dots&0\\
\end{bmatrix}.
\end{aligned}
\end{equation}
We also refer 
to $\eqref{deformlin}$ as the 
linearization of a perturbed matrix polynomial. 
Recall that, an $m \times n$ matrix pencil $\lambda A_1 + A_0$ is called {\it strictly equivalent} to $\lambda B_1 + B_0$ if there are non-singular matrices $Q$ and $R$ such that $Q^{-1}A_1R =B_1$ and $Q^{-1}A_0R=B_0$. Note that two $m \times n$ matrix pencils have the same complete eigenstructure if and only if they are strictly equivalent. Moreover, two $m \times n$ matrix polynomials of degree $d$, $P(\lambda)$ and $Q(\lambda)$, have the same complete eigenstructure if and only if ${\cal C}^1_{P(\lambda)}$ and ${\cal C}^1_{Q(\lambda)}$ are strictly equivalent. Now we can state one of our goals as finding a perturbation $E(\lambda)$ such that ${\cal C}^1_{P(\lambda)} + {\cal E}$ and ${\cal C}^1_{P(\lambda)+ E(\lambda)}$ are strictly equivalent. The existence of such a perturbation $E(\lambda)$ is known and 
stated in Theorem \ref{deform}, which is a simplified version of Theorem~5.21 in \cite{DLPV18}, it is also a slightly adapted formulation of Theorem~2.5 in \cite{DmDo17}, see also \cite{DJKV19,JoKV13,VaDe83}. 

\begin{theorem} \label{deform}%
  Let $P(\lambda)$ be an $m\times n$ matrix polynomial of degree $d$ and let
  ${\cal C}^1_{P(\lambda)}$ be its first companion form. If ${\cal
      C}^1_{P(\lambda)} + {\cal E}$ is a perturbation of ${\cal C}^1_{P(\lambda)}$ such
  that
  \[
   ||{\cal E}|| = || ({\cal C}^1_{P(\lambda)} + {\cal E}) - {\cal C}^1_{P(\lambda)} || < \frac{\pi}{12 \, d^{3/2}} \, ,
  \]
  then ${\cal C}^1_{P(\lambda)} + {\cal E}$ is strictly equivalent to the linearization of the perturbed polynomial ${\cal C}^1_{P(\lambda)+ E(\lambda)}$, i.e. there exist two nonsingular
  matrices $U$ and $V$ (they are small perturbations of the identity matrices)
  such that
  \begin{equation*} 
    U \cdot ({\cal C}^1_{P(\lambda)} + {\cal E}) \cdot V = {\cal C}^1_{P(\lambda)+ E(\lambda)},  
  \end{equation*}  
  moreover, 
  \[
    || {\cal C}^1_{P(\lambda) + E(\lambda)} - {\cal C}^1_{P(\lambda)} || \leq 4 \, d \, (1+||P(\lambda)||_F) \;  || {\cal E} || \, .
  \]
\end{theorem} 
Theorem \ref{deform} guaranties the existence of the structured perturbation \eqref{deformlin} and the transformation matrices $U$ and $V$. 
In the following section we present an algorithm that, for a given perturbation \eqref{fulldeform}, finds such a structured perturbation \eqref{deformlin}, 
and transformation matrices explicitly.  

\section{Reduction algorithm} 
In this section we describe our algorithm that by strict equivalence transformation reduces a full perturbation of a linearization pencil \eqref{fulldeform} to a structured perturbation of this pencil \eqref{deformlin}, i.e. a perturbation where only the blocks that correspond to the matrix coefficients of a matrix polynomial are perturbed.  
The corresponding transformation matrices are derived too.  
We also analyze important characteristics of the proposed algorithm and its outputs.

Define an unstructured perturbation ${\cal E}^u$ of the linearization ${\cal C}^1_{P(\lambda)}$ as a perturbation \eqref{fulldeform} where the blocks $E_{11}, E_{11}', E_{12}', \dots, E_{1d}'$ are substitutet with the zero blocks of the corresponding sizes. ${\cal E}^u$ consists of all the perturbation blocks that are not included in the structured perturbation \eqref{deformlin}, i.e. ${\cal E}^u$ consists of all the perturbations of the identity and zero blocks of the linearization ${\cal C}^1_{P(\lambda)}$.

In Section \ref{secconv} 
we show that unstructured part of perturbation tends to zero (entry-wise) as the number of iterations grows; in Section \ref{secbounds} we derive a bound on the norm of the resulting structured perturbation; in Section \ref{sectr} we explain how to construct the corresponding transformation matrices, i.e. matrices that reduce a full perturbation to a structured one. 

We note that the construction the corresponding transformation matrices in this paper is similar to the construction of the transformation matrices for the reduction to miniversal deformations of matrices in \cite{DmFS12,DmFS14}, as well as that the evaluation of the norm of the structured part has some similarities with the evaluation of the norm of the miniversal deformation of (skew-)symmetric matrix pencils in \cite{Dmyt16,Dmyt19}, see also \cite{DmFS12,DmFS14}. These similarities are due to the fact that our structured perturbation is in fact a versal deformation (but not miniversal), see the mentioned papers for the definitions and details. 


\begin{algorithm} 
\label{alg}
Let ${\cal C}^1_{P(\lambda)}$ be a first companion linearization of a matrix polynomial $P(\lambda)$ and ${\cal E}_1$ be a full perturbation of ${\cal C}^1_{P(\lambda)}$. 
\begin{itemize}
\item[] {Input:} Matrix polynomial $P(\lambda)$, perturbed matrix pencil ${\cal C}^1_{P(\lambda)}+{\cal E}_1$, and the tolerance parameter ${\rm tol}$;
\vskip0.2cm
\item[] {Initiation:} $U_1:=I$ and $V_1:=I$ 
\vskip0.2cm
\item[] {Computation:} While $|| {\cal E}^u_i || > {\rm tol}$ 
\begin{itemize}
\vskip0.2cm
\item solve the coupled Sylvester equations: \\ $\left(({\cal C}^1_{P(\lambda)} + {\cal E}_{i})X + Y({\cal C}^1_{P(\lambda)} + {\cal E}_{i})\right)^u = - {\cal E}^u_i$;
\vskip0.2cm
\item update the perturbation of the linearization: \\ $({\cal C}^1_{P(\lambda)} + {\cal E}_{i+1}):=(I+Y) ({\cal C}^1_{P(\lambda)} + {\cal E}_{i}) (I+X)$;
\vskip0.2cm
\item update the transformation matrices: \\ $U_{i+1} := (I+Y)U_{i}$ and $V_{i+1} := V_{i}(I+X)$; 
\vskip0.2cm
\item extract the new unstructured perturbation ${\cal E}^u_{i+1} $ to be eliminated; 
\vskip0.2cm
\item increase the counter: \ $i:=i+1$;
\end{itemize}
\vskip0.2cm
\item[] {Output:} Structurally perturbed linearization pencil ${\cal C}^1_{P(\lambda)+E(\lambda)}:={\cal C}^1_{P(\lambda)} + {\cal E}_{k}$, where ${\cal E}_{k}$ is a structured perturbation (since the norm of $||{\cal E}^u_{k}|| < {\rm tol}$), and the transformation matrices $U$ and $V$. 
\end{itemize}
\end{algorithm}
In the rest of the paper we investigate properties of this algorithm and perform numerical experiments. 
\subsection{Elimination of unstructured perturbation} 
\label{secconv}
We start by deriving an auxiliary lemma that will be used to prove that following Algorithm \ref{alg} results in a convergence of the unstructured perturbation to zero.  

For a given matrix $T$, define $\kappa(T):= \kappa_F(T) = ||T|| \cdot ||T^{\dagger}||$ to be a Frobenius condition number of $T$, see e.g., \cite{AvDT13, ChRa08, SuGo13}. 

\begin{lemma}
\label{nplemma}
Let $A,B,C,D,M$, and $N$ be $m\times n$ matrices and let $X$ and $Y$ be $n \times n$ and $m \times m$ matrices, respectively, that are the smallest norm solution of the system of coupled Sylvester equations 
\begin{equation}
\label{sysABCD}
\begin{aligned}
AX + YB&= M,\\
CX + YD&= N.
\end{aligned}
\end{equation}
Then 
\begin{equation}
\label{probbound}
||X|| \cdot ||Y|| \leq  \frac{\kappa(T)^2}{2(n||A||^2 + m ||B||^2 + n ||C||^2 + m ||D||^2)} \left(||M||^2 + ||N||^2\right), 
\end{equation}
where $T=\begin{bmatrix}
I_n \otimes A & B^T \otimes I_m \\
I _n\otimes C & D^T \otimes I_m 
\end{bmatrix}$ is the Kronecker product matrix associated with the system \eqref{sysABCD}. 
\end{lemma}
\begin{proof}
Using Kronecker product we can rewrite the system of coupled Sylvester equations as a system of linear equations $Tx=b$, or explicitly 
\begin{equation}
\label{kron}
\begin{bmatrix}
I_n \otimes A & B^T \otimes I_m \\
I _n\otimes C & D^T \otimes I_m 
\end{bmatrix}
\begin{bmatrix}
\vv (X) \\
\vv (Y)
\end{bmatrix} = 
\begin{bmatrix}
\vv (M) \\
\vv (N)
\end{bmatrix}. 
\end{equation}
The least-squares solution of the smallest norm of such system can be written as $x = T^{\dagger} b$, implying $||x|| \leq ||T^{\dagger}|| \cdot ||b||$ or more explicitly, and taking into account $||x||^2 = ||X||^2 + ||Y||^2$:  
\begin{equation}
\begin{aligned}
\label{bound1}
||X||^2 + ||Y||^2 &\leq ||T^{\dagger}||^2 \left(||M||^2 + ||N||^2\right)= \frac{\kappa(T)^2}{||T||^2} \left(||M||^2 + ||N||^2\right) \\
& = \frac{\kappa(T)^2}{\left(n||A||^2 + m ||B||^2 + n ||C||^2 + m ||D||^2\right)} \left(||M||^2 + ||N||^2\right),
\end{aligned}
\end{equation}
where $\kappa(T)$ is the Frobenius condition number of $T$. 
Taking into account that 
\begin{equation*}
||X|| \cdot ||Y|| \leq \frac{1}{2}\left(||X||^2 + ||Y||^2\right),
\end{equation*}
we obtain 
\begin{equation*}
||X|| \cdot ||Y|| \leq \frac{\kappa(T)^2}{2(n||A||^2 + m ||B||^2 + n ||C||^2 + m ||D||^2)} \left(||M||^2 + ||N||^2\right).
\end{equation*}
\end{proof}
The bounding expression in \eqref{probbound} depends on the conditioning of the problem \eqref{kron} as well as on how small (normwise) the right-hand-side of \eqref{kron} (or, respectively, \eqref{sysABCD}) is, comparing to the matrix coefficients in the left-hand-side. The conditioning of \eqref{sysABCD} may actually be better than the conditioning of \eqref{kron}. Thus for very ill-conditioned problems and large perturbations, it may be reasonable to use a solver 
for \eqref{sysABCD} instead of passing to the Kronecker product matrices. 

In the following theorem we prove that Algorithm \ref{alg} eliminates the unstructured perturbation, i.e. we show that the norm of the unstructured part of the perturbation tends to zero as the number of iterations grows.

\begin{theorem} 
Let ${\cal C}^1_{P(\lambda)} + {\cal E}_{1}$ be a perturbation of the linearization and let $\alpha ||{\cal E}_{1}|| < 1$, where $\alpha = \alpha ({\cal C}^1_{P(\lambda)},{\cal E}_{1})$ is defined in \eqref{aldef}, 
then $||{\cal E}_{i}^u|| \rightarrow 0$ if $i \rightarrow \infty$. 
\end{theorem}
\begin{proof}
We start by proving a bound for the norm of the unstructured part of a perturbation at the $(i+1)$-st step of the algorithm, using the norm of the unstructured part of a perturbation at the $i$-th step of the algorithm. Define ${\cal C}^1_{P(\lambda)} = \lambda W + \widetilde{W}$. 

Following Algorithm \ref{alg} we obtain matrices $X_i$ and $Y_i$ by solving the system of coupled Sylvester matrix equations: 
\begin{equation}
\label{initeq}
\begin{split}
\left((W+E_i)X_i + Y_i(W+E_i) \right)^u&=E_{i}^u,\\
\left((\widetilde{W}+\widetilde{E}_i)X_i + Y_i(\widetilde{W}+\widetilde{E}_i)  \right)^u&= \widetilde{E}_{i}^u.
\end{split}
\end{equation}
Using the solution $X_i$ and $Y_i$ of the system \eqref{initeq} we compute
\begin{align*}
W+E_{i+1} &:= (I+ Y_i)(W+E_i)(I+ X_i),\\
\widetilde{W}+\widetilde{E}_{i+1} &:= (I+ Y_i)(\widetilde{W}+\widetilde{E}_i)(I+ X_i), 
\end{align*}
%
%
or equivalently, 
\begin{align*}
E_{i+1} &:=E_{i}^s + \left(E_{i}^u + (W+E_i)X_i + Y_i(W+E_i) \right) + Y_i(W+E_i)X_i,\\
\widetilde{E}_{i+1} &:= \widetilde{E}_{i}^s + \left(\widetilde{E}_{i}^u  + (\widetilde{W}+\widetilde{E}_i)X_i + Y_i(\widetilde{W}+\widetilde{E}_i) \right) + Y_i(\widetilde{W}+\widetilde{E}_i)X_i.
\end{align*}
Since $X_i$ and $Y_i$ are a solution of \eqref{initeq} we have
\begin{align*}
E_{i+1} & = E_{i}^s + \left((W+E_i)X_i + Y_i(W+E_i) \right)^s + Y_i(W+E_i)X_i,\\
\widetilde{E}_{i+1} &= \widetilde{E}_{i}^s + \left((\widetilde{W}+\widetilde{E}_i)X_i + Y_i(\widetilde{W}+\widetilde{E}_i)  \right)^s +Y_i(\widetilde{W}+\widetilde{E}_i)X_i. 
\end{align*}
Splitting the perturbation into the structured and unstructured parts we obtain  
\begin{align*}
E_{i+1}^s &= E_{i}^s + \left((W+E_i)X_i + Y_i(W+E_i) \right)^s + \left(Y_i(W+E_i)Y_i\right)^s,\\
\widetilde{E}_{i+1}^s  &= \widetilde{E}_{i}^s + \left((\widetilde{W}+\widetilde{E}_i)X_i + Y_i(\widetilde{W}+\widetilde{E}_i)  \right)^s + \left(Y_i(\widetilde{W}+\widetilde{E}_i)X_i\right)^s, \\
E_{i+1}^u &= \left(Y_i(W+E_i)X_i\right)^u,\\
\widetilde{E}_{i+1}^u  &= \left(Y_i(\widetilde{W}+\widetilde{E}_i)X_i\right)^u.
\end{align*}
In general, $E_{i+1}^u$ and $\widetilde{E}_{i+1}^u$ are not zero matrices but we show that they tend to zero (entry-wise) when $i \rightarrow \infty$. Using the bound \eqref{probbound} on the Frobenious norm of $||E_{i+1}^u||$ we have:
\begin{align*}
||E_{i+1}^u|| & \leq ||\left(Y_i(W+E_i)X_i\right)^u|| \leq ||Y_i(W+E_i)X_i|| \leq ||X_i||\cdot ||Y_i|| \cdot ||W+E_i|| \\
& \leq \frac{\kappa(T_i)^2 ||W+E_i||}{2\left(n+m\right) \left(||W+E_i||^2 + ||\widetilde{W}+\widetilde{E}_i||^2\right)} \ ||{\cal E}_i^u||^2, 
\end{align*}
similarly, for the matrix $||\widetilde{E}_{i+1}^u||$,  
\begin{equation}
\label{unbound}
\begin{aligned}
||\widetilde{E}_{i+1}^u|| & \leq ||\left(Y_i(\widetilde{W}+\widetilde{E}_i)X_i\right)^u|| \leq ||Y_i(\widetilde{W}+\widetilde{E}_i)X_i|| \leq ||X_i ||\cdot ||Y_i|| \cdot ||\widetilde{W}+\widetilde{E}_i|| \\
& \leq \frac{\kappa(T_i)^2 ||\widetilde{W}+\widetilde{E}_i||}{2\left(n+m\right) \left(||W+E_i||^2 + ||\widetilde{W}+\widetilde{E}_i||^2\right)} \ ||{\cal E}_i^u||^2, 
\end{aligned}
\end{equation}
where 
\begin{equation}
\label{kronWW}
T_i=\begin{bmatrix}
I_{nd} \otimes (W+E_i) & (W+E_i)^T \otimes I_{m+n(d-1)} \\
I _{nd}\otimes (\widetilde{W}+\widetilde{E}_i) & (\widetilde{W}+\widetilde{E}_i)^T \otimes I_{m+n(d-1)} 
\end{bmatrix}
\end{equation} 
is the Kronecker product matrix associated with the system of coupled Sylvester equations \eqref{initeq}. 

Define $\alpha$ as follows   
\begin{equation}
\label{aldef}
\begin{aligned}
\hskip-1cm\alpha:= \sup_{i} \left\{ \frac{\kappa(T_i)^2 ||W+E_i||}{\sqrt{2}\left(n+m\right) \left(||W+E_i||^2 + ||\widetilde{W}+\widetilde{E}_i||^2\right)},\right. \qquad\qquad\qquad\\ 
\qquad\qquad\qquad\qquad\qquad\qquad \left. \frac{\kappa(T_i)^2 ||\widetilde{W}+\widetilde{E}_i||}{\sqrt{2}\left(n+m\right) \left(||W+ E_i||^2 + ||\widetilde{W}+\widetilde{ E}_i||^2\right)} 
\right\}. 
\end{aligned}
\end{equation}
Here we assume that our initial perturbation is such that $\kappa(T_i)$ does not change much and thus the supremum in the definition of $\alpha$ \eqref{aldef} is finite.  
Now the bounds on the unstructured part of the perturbation for the both matrices of the matrix pencil at the step $i+1$ can be written as follows
\begin{equation}
\label{alpbo1}
||E_{i+1}^u||  \leq \frac{\alpha}{\sqrt{2}} ||{\cal E}^u_i||^2 \quad \text{ and } \quad 
||\widetilde{E}_{i+1}^u||  \leq \frac{\alpha}{\sqrt{2}} ||{\cal E}^u_i||^2.
\end{equation}
This results into the bound for the whole pencil:
\begin{align}
\label{alpbo2}
||{\cal E}_{i+1}^u|| = \left(||E_{i+1}^u||^2 + ||\widetilde{E}_{i+1}^u||^2\right)^{\frac{1}{2}} 
& \leq \alpha ||{\cal E}_i^u||^2. 
\end{align}
Using the bounds \eqref{alpbo1} and \eqref{alpbo2} at each step we get 
\begin{equation}
\label{alpbo3}
\max \left\{ ||E_{k}^u||, ||\widetilde{E}_{k}^u|| \right\} \leq \left(\frac{\alpha}{\sqrt{2}}\right)^{2^{k-1}-1} ||{\cal E}_{1}^u||^{2^{k-1}} \text{ \ \ and \ \ } 
||{\cal E}_{k}^u|| \leq \alpha^{2^{k-1}-1} ||{\cal E}_{1}^u||^{2^{k-1}} . 
\end{equation}
If $\alpha ||{\cal E}_{1}|| < 1$ then the norm of the unstructured part of the perturbation tends to zero with the iteration grows. 
\end{proof}

\begin{remark}
In our case we should exclude some rows from \eqref{kronWW}, since we want to eliminate only the unstructured part of the perturbation ${\cal E}_i$. Therefore the norm of the solution of such the least-squares problem will be less than or equal to $||x||$, where $x= T^{\dagger} b$. Clearly, the bounds from Lemma~\ref{nplemma} remain valid.
\end{remark}
%
%
The sharpness of the bounds \eqref{alpbo3} depends on the value of $\alpha$ and on the size of an initial perturbation: the better conditioned the problem is and the smaller initial perturbation is, the better the bounds \eqref{alpbo3} are. Even if the problem is ill-conditioned we can still guarantee the convergence for small enough perturbations. Note that, a proper scaling of a matrix polynomial improves the conditioning of the problem, see e.g., \cite{DLPV18}. 
Moreover,  in practice, Algorithm \ref{alg} converges to a structured perturbation very well and requires only a small number of iterations, see the numerical experiments in Section~\ref{numexp}. 

\subsection{Bound on the norm of structured perturbation}
\label{secbounds}
In this section we find a bound on the resulting structured perturbation.
%
%
Similarly to the analysis in Section \ref{secconv} we have a dependency on the conditioning of the problem as well as on the norm of an original perturbation. Therefore we need to make an assumption that these quantities 
are small enough. 
\begin{theorem}
\label{thbounds}
Let ${\cal C}^1_{P(\lambda)} + {\cal E}_{1}$ be a perturbation of the linearization ${\cal C}^1_{P(\lambda)}$, $||{\cal E}_1|| = \varepsilon $, and $\alpha \varepsilon < 1$, where $\alpha = \alpha ({\cal C}^1_{P(\lambda)},{\cal E}_{1})$ 
is defined in \eqref{aldef}. Define also 
%
$
\beta:= \ \sup_{i} \ 
 \sqrt{\frac{2}{\left(n+m\right)}} \kappa(T_i)$ for the Kronecker product matrix $T_i$, see \eqref{kronWW}. Then $||{\cal E}^s|| < \varepsilon (1+ \beta) / (1 - \alpha \varepsilon)$. 

\end{theorem}
\begin{proof}
For the input ${\cal C}^1_{P(\lambda)}+{\cal E}_1$, following Algorithm \ref{alg} step-by-step, we can build the resulting structured perturbation as follows:  
\begin{equation*}
\begin{aligned}
{\cal C}^1_{P(\lambda)}+{\cal E}^s &= {\cal C}^1_{P(\lambda)}+{\cal E}_1^s + \left(({\cal C}^1_{P(\lambda)}+{\cal E}_1)X_1 + Y_1({\cal C}^1_{P(\lambda)}+{\cal E}_1) \right)^s + \left(X_1({\cal C}^1_{P(\lambda)}+{\cal E}_1)Y_1\right)^s\\
&+ \left(({\cal C}^1_{P(\lambda)}+{\cal E}_2)X_2 + Y_2({\cal C}^1_{P(\lambda)}+{\cal E}_2) \right)^s  + \left(X_2({\cal C}^1_{P(\lambda)}+{\cal E}_2)Y_2\right)^s + \ldots \\ 
\ldots &+ \left(({\cal C}^1_{P(\lambda)}+{\cal E}_i)X_i + Y_i({\cal C}^1_{P(\lambda)}+{\cal E}_i) \right)^s   + \left(X_{i}({\cal C}^1_{P(\lambda)}+{\cal E}_{i})Y_{i}\right)^s + \dots \\ 
\end{aligned}
\end{equation*}
%
%

We start by evaluating the structured part of the perturbation coming from the coupled Sylvester equations: 
\small{
\begin{equation}
\label{evpart}
\begin{aligned}
&||\left(({\cal C}^1_{P(\lambda)}+{\cal E}_i)X_i + Y_i({\cal C}^1_{P(\lambda)}+{\cal E}_i) \right)^s|| 
\\ 
&\leq \sqrt{ \frac{2\kappa(T_i)^2 ||W+E_i||^2}{\left(n+m\right) \left(||W+E_i||^2 + ||\widetilde{W}+\widetilde{E}_i||^2\right)}
+ 
\frac{2\kappa(T_i)^2 ||\widetilde{W}+\widetilde{E}_i||^2}{\left(n+m\right) \left(||W+E_i||^2 + ||\widetilde{W}+\widetilde{E}_i||^2\right)} } || {\cal E}_i^u || \\
&= \frac{\sqrt{2}\kappa(T_i) \sqrt{||W+E_i||^2 + ||\widetilde{W}+\widetilde{E}_i||^2 } }{\sqrt{\left(n+m\right) \left(||W+E_i||^2 + ||\widetilde{W}+\widetilde{E}_i||^2 \right)}} \ || {\cal E}_i^u ||
\leq \ \beta \ || {\cal E}_i^u ||.
\end{aligned}
\end{equation}}
\noindent Recall that our initial perturbation is such that $\kappa(T_i)$ does not change much and thus the supremum in the definition of $\beta$ is finite. 
Note that the bounds for $||{\cal E}^u_i||$ are also bounds for $||X_{i}({\cal C}^1_{P(\lambda)}+{\cal E}_{i})Y_{i}||$, see \eqref{unbound}, and thus also for $||\left(X_{i}({\cal C}^1_{P(\lambda)}+{\cal E}_{i})Y_{i}\right)^s||$. Thus we can evaluate the norm of $||{\cal E}^s_i||$ using \eqref{alpbo3} and \eqref{evpart} as well as noting that $||{\cal E}^s_i||$ and $||{\cal E}^u_i||$ are less than or equal to $||{\cal E}_i||$:
\begin{equation}
\label{summation}
\begin{aligned}
||{\cal E}^s|| = ||{\cal E}_1^s|| &+ 
 ||\left(({\cal C}^1_{P(\lambda)}+{\cal E}_1)X_1 + Y_1({\cal C}^1_{P(\lambda)}+{\cal E}_1) \right)^s|| + 
||\left(X_1({\cal C}^1_{P(\lambda)}+{\cal E}_1)Y_1\right)^s|| + \\
&+ ||\left(({\cal C}^1_{P(\lambda)}+{\cal E}_2)X_2 + Y_2({\cal C}^1_{P(\lambda)}+{\cal E}_2) \right)^s|| + 
||\left(X_2({\cal C}^1_{P(\lambda)}+{\cal E}_2)Y_2\right)^s|| + \\
\ldots &+ ||\left(({\cal C}^1_{P(\lambda)}+{\cal E}_i)X_i + Y_i({\cal C}^1_{P(\lambda)}+{\cal E}_i) \right)^s|| + 
||\left(X_{i}({\cal C}^1_{P(\lambda)}+{\cal E}_{i})Y_{i}\right)^s|| + \dots \\
&= \varepsilon 
+ {\beta} ||{\cal E}_1^u|| + \alpha ||{\cal E}_1^u||^2 
+ {\beta} ||{\cal E}_2^u||  + \alpha ||{\cal E}_2^u||^2 + 
 \ldots + {\beta} ||{\cal E}_i^u|| + \alpha ||{\cal E}_i^u||^2 + \dots  \\
&= \varepsilon 
+ {\beta} \varepsilon + \alpha \varepsilon^2 
+ {\beta} \alpha \varepsilon^2  + \alpha^3 \varepsilon^4 
+ {\beta} \alpha^3 \varepsilon^4  + \alpha^7 \varepsilon^8 + \\
 \ldots & + {\beta} \alpha^{2^{i-1}-1} \varepsilon^{2^{i-1}} + \alpha^{2^{i}-1}\varepsilon^{2^{i}} + \dots  \\
&=   \varepsilon (1 + {\beta}) \left( 1 + \alpha \varepsilon + (\alpha \varepsilon)^3 + \ldots + (\alpha \varepsilon)^{2^{i-1}-1} + \dots  \right)  \\
&=  \varepsilon (1 + {\beta}) \left( \sum_{i=0}^{\infty}(\alpha \varepsilon)^{2^i-1} \right) \leq  \frac{ \varepsilon (1 + {\beta}) }{1 - \alpha \varepsilon}.\\
\end{aligned}
\end{equation}
\end{proof}
The bound in Theorem \ref{thbounds} is not very tight if $\alpha \varepsilon$ is close to 1 but it is quite good for small $\alpha \varepsilon$. For example, Theorem \ref{thbounds} says that for $\alpha \varepsilon < 1/n$ we get $||{\cal E}^s|| \leq n \varepsilon (1 + {\beta}) / (n-1)$, and in particular, for $\alpha \varepsilon < 1/2$ we get $||{\cal E}^s|| \leq 2 (1 + {\beta}) \varepsilon$.

\subsection{Construction of the transformation matrices}
\label{sectr}
In this section we investigate the transformation matrices that bring a full perturbation of the linearization to a structured perturbation of the linearization. Following Algorithm \ref{alg}, we observe that the transformation matrices are constructed as the following infinite products: 
\begin{equation*}
U=\lim_{i \to \infty} (I_{m}+X_i) \cdots (I_{m}+X_2)(I_{m}+X_1) \text{ and } 
V=\lim_{i \to \infty} (I_{n}+Y_1)(I_{n}+Y_2) \cdots  (I_{n}+Y_i). 
\end{equation*}
Convergence of these infinite products to nonsingular matrices is proven in Theorem \ref{trmat}. Note that, for the small initial perturbations such transformation matrices are small perturbations of the identity matrices. 
\begin{theorem}
\label{trmat}
Let ${\cal C}^1_{P(\lambda)} + {\cal E}_{1}$ be a perturbation of the linearization ${\cal C}^1_{P(\lambda)}$, and $\alpha ||{\cal E}_1|| < 1$, where $\alpha = \alpha ({\cal C}^1_{P(\lambda)},{\cal E}_{1})$ 
is defined in \eqref{aldef}. Let also $X_i$ and $Y_i$ be a solution of \eqref{initeq} for the corresponding index $i$, and $I_m$ and $I_n$ be the $m \times m$ and $n \times n$ identity matrices. Then 
\begin{equation}\label{30z}
\lim_{i \to \infty} (I_{m}+X_i) \cdots (I_{m}+X_2)(I_{m}+X_1) \text{ and } 
\lim_{i \to \infty} (I_{n}+Y_1)(I_{n}+Y_2) \cdots  (I_{n}+Y_i)
\end{equation}
exist and are nonsingular matrices. 
\end{theorem}
\begin{proof}
\hide{
For each natural number $l$, put
\begin{equation}\label{28z}
U_{i+1}:= (I_{n}+X_i) \cdots (I_{n}+X_2)(I_{n}+X_1).
\end{equation}
Analogously for the second matrix
\begin{equation}\label{29z}
V_{i+1}:= (I_{n}+Y_1)(I_{n}+Y_2) \cdots  (I_{n}+Y_i).
\end{equation}

Now let $i \rightarrow \infty$ then
\begin{equation}\label{30z}
\lim_{i \rightarrow \infty} U_{i+1}:= \cdots (I_{n}+X_i) \cdots (I_{n}+X_2)(I_{n}+X_1), \text{ and } 
\lim_{i \rightarrow \infty} V_{i+1}=\prod_{i=1}^{\infty}(I_{n}+Y_i).
\end{equation}
}
By \cite[Theorem 4]{Tren99} the limits in \eqref{30z} exist and are nonsingular matrices 
if the sums 
\begin{equation}
\label{sums}
\|X_1\|+\|X_2\|+\|X_3\|+\cdots= \sum_{i=1}^{\infty} \|X_i\| \quad \text{ and }  \quad 
\|Y_1\|+\|Y_2\|+\|Y_3\|+\cdots= \sum_{i=1}^{\infty} \|Y_i\|, 
\end{equation}
respectively, absolutely converge. 

Using the bound \eqref{bound1} for a solution of coupled Sylvester equations and noting that $||X||^2 \leq ||X||^2 + ||Y||^2$, we have the following bound for each $||X_{i}||^2$ and $||Y_{i}||^2$: 
\begin{equation}
\begin{aligned}
\label{XYb}
||X_{i}||  &\leq \frac{\kappa(T_i)^2 }{\left(n+m\right) \left(||W+E_i||^2 + ||\widetilde{W}+\widetilde{E}_i||^2\right)} ||{\cal E}^u_i||^2 \leq 
\alpha ||{\cal E}^u_i||^2  \leq \alpha^{2^{i-1}-1} ||{\cal E}_1^u||^{2^{i-1}}, \\ 
||Y_{i}||  &\leq \frac{\kappa(T_i)^2 }{\left(n+m\right) \left(||W+E_i||^2 + ||\widetilde{W}+\widetilde{E}_i||^2\right)} ||{\cal E}^u_i||^2 \leq
\alpha ||{\cal E}^u_i||^2 \leq \alpha^{2^{i-1}-1} ||{\cal E}_1^u||^{2^{i-1}}.
\end{aligned}
\end{equation}
Bounds \eqref{XYb} allow us to use \eqref{summation}, and conclude that the sums in \eqref{sums} absolutely converge for $\alpha ||{\cal E}_1|| < 1, \ (||{\cal E}_1^u|| \leq ||{\cal E}_1||)$. 
\end{proof}

\hide{
\begin{lemma}\label{lem3c}
Let $A\in{\mathbb C}^{\,n\times n}$ and
let $\cal D$ be a $(0,\!*)$ matrix of
size $n\times n$ satisfying
\eqref{a4nf}. Let $m$ be the natural
number from Lemma \ref{lem1}, we
suggest that $m\ge 3$,
and let $M$ be any $n$-by-$n$ matrix
$M$ satisfying $\|M\|<m^{-8}$. Then
there exists a matrix $S=I_n+X$
depending holomorphically on the
entries of $M$ in a neighborhood of
zero such that
\[S^{T}(A+M)S-A\in {\cal
D}({\mathbb C})\] and
$S=I_n$ if $M=0$.
\end{lemma}

absolutely converges due to \eqref{5aa}
and \eqref{24z} thus by \cite[Theorem 4]{inf} the
product \eqref{28z} converges to some invertible matrix $S$.
The entries of $S$ are holomorphic functions in
the entries of $M$ (that satisfies
$\|M\|<m^{-8}$).
}

\section{Numerical experiments}
\label{numexp}
All the numerical experiments are performed on MacBook Pro (processor: 2,6 GHz Intel Core i7, memory: 32 GB 2400 MHz DDR4), using Matlab R2019a (64-bit). We consider a large number of randomly generated matrix polynomials, matrix polynomials coming from real world applications, and specially crafted matrix polynomials for testing the limits of the proposed algorithm.    

\begin{example}
\label{ex1}
Consider 1000 random polynomials of the size $3 \times 3$ and degree $5$. The entries of the matrix coefficients of these polynomials are generated from the normal distribution with the mean $\mu = 0$ and standard deviation $\sigma=10$ (variance $\sigma^2=100$). The polynomials are normalized to have the Frobenius norm equal to $1$. Each polynomial is perturbed by adding a matrix polynomial whose matrix coefficients have entries that are uniformly distributed numbers on the interval $(0,0.1)$. At most 6 iterations are needed for the norm of the unstructured part of a perturbation to be smaller than $10^{-16}$ ($10^{-16}$ is the tolerance we require). In Figure \ref{rex1flog} we present the results in whisker plots (box plots). 
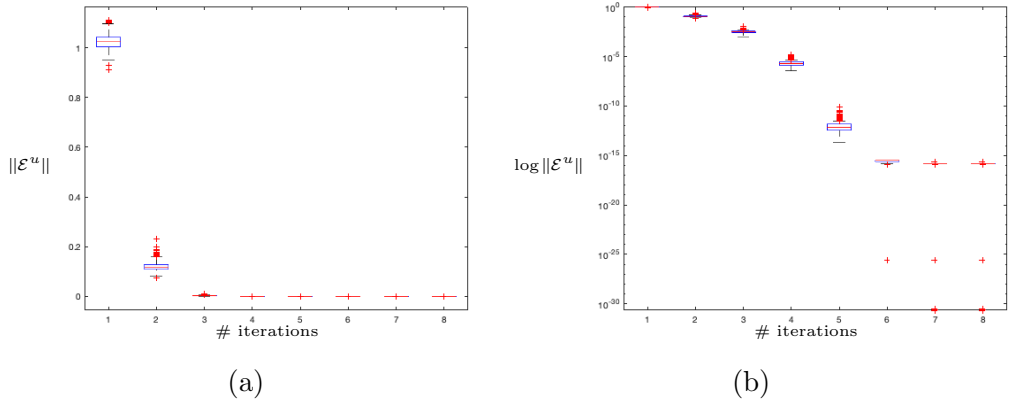
\begin{figure}[h!]
  \centering
\begin{subfigure}[b]{0.48\textwidth}    
\begin{tikzpicture}
  \sbox0{\includegraphics[width=1.0\textwidth]{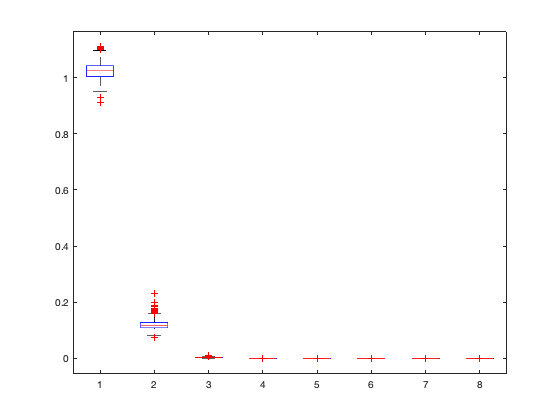}}
  \node[above right,inner sep=0pt] at (0,0)  {\usebox{0}};
  \node[black] at (0.5\wd0,0.05\ht0) {{\tiny\# iterations}};
  \node[black] at (0.02\wd0,0.5\ht0) {{\tiny$||{\cal E}^u||$}};
\end{tikzpicture}
 \caption{}
  \end{subfigure}
  \begin{subfigure}[b]{0.48\textwidth}    
    \begin{tikzpicture}
  \sbox0{\includegraphics[width=1.0\textwidth]{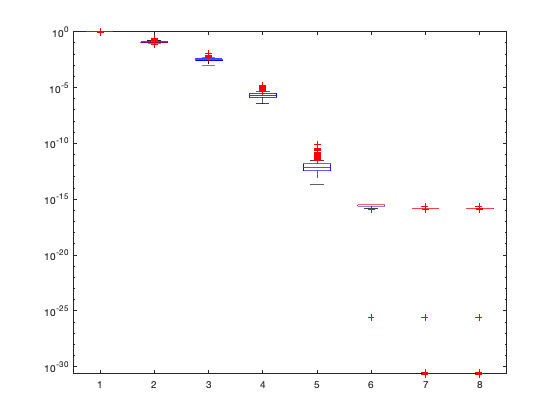}}
  \node[above right,inner sep=0pt] at (0,0)  {\usebox{0}};
  \node[black] at (0.5\wd0,0.05\ht0) {{\tiny\# iterations}};
  \node[black] at (-0.02\wd0,0.5\ht0) {{\tiny$\log ||{\cal E}^u||$}};
\end{tikzpicture}
  \caption{}
  \end{subfigure}
\caption{The whisker plots illustrates the elimination of the unstructured part of the perturbation 
for 1000 perturbed random polynomials of the size $3 \times 3$ and degree $5$. In (a) the Frobenius norms of unstructured parts of perturbations are plotted on the y-axis and the iterations are on the x-axis. In (b) the same data is presented in a whisker plot with a logarithmic scale on the y-axis. 
}
    \label{rex1flog}
\end{figure}
\end{example}

\begin{example}
Consider 1000 random polynomials of the size $8 \times 8$ and degree $4$. The entries of the matrix coefficients of these polynomials are generated from the normal distribution with the mean $\mu = 0$ and the standard deviation $\sigma=10$ (variance $\sigma^2=100$). These polynomials are normalized and perturbed as in Example \ref{ex1}. Once again at most 6 iterations are needed for the norm of the unstructured part of a perturbation to be of order $10^{-16}$. In Figure \ref{rex2flog} we present the results in whisker plots (box plots). 
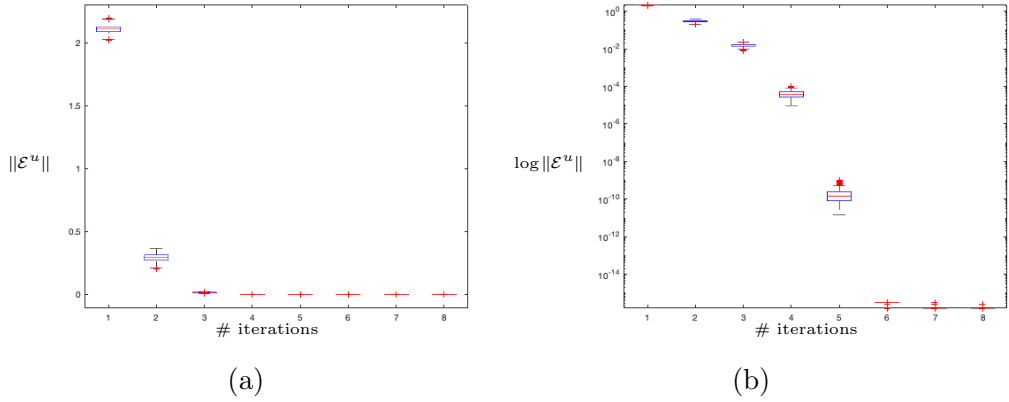
\begin{figure}[h!]
  \centering
    \begin{subfigure}[b]{0.48\textwidth}    
\begin{tikzpicture}
  \sbox0{\includegraphics[width=1.0\textwidth]{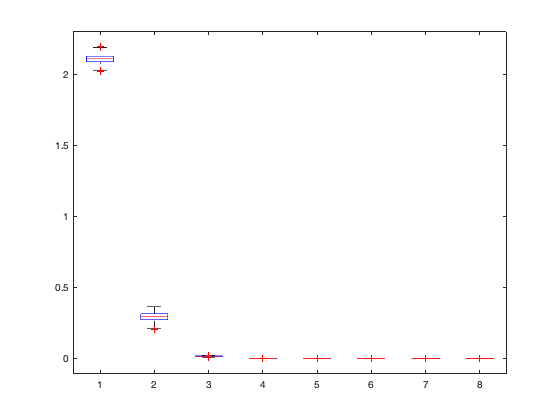}}
  \node[above right,inner sep=0pt] at (0,0)  {\usebox{0}};
  \node[black] at (0.5\wd0,0.05\ht0) {{\tiny\# iterations}};
  \node[black] at (0.02\wd0,0.5\ht0) {{\tiny$||{\cal E}^u||$}};
\end{tikzpicture}
 \caption{}
  \end{subfigure}
  \begin{subfigure}[b]{0.48\textwidth}    
    \begin{tikzpicture}
  \sbox0{\includegraphics[width=1.0\textwidth]{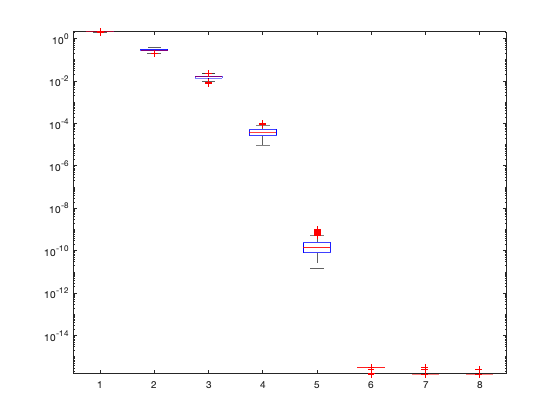}}
  \node[above right,inner sep=0pt] at (0,0)  {\usebox{0}};
  \node[black] at (0.5\wd0,0.05\ht0) {{\tiny\# iterations}};
  \node[black] at (-0.02\wd0,0.5\ht0) {{\tiny$\log ||{\cal E}^u||$}};
\end{tikzpicture}
  \caption{}
  \end{subfigure}
\caption{The whisker plots illustrates the elimination of the unstructured part of the perturbation 
for 1000 perturbed random polynomials of the size $8 \times 8$ and degree $4$. In (a) the Frobenius norms of unstructured parts of perturbations are plotted on the y-axis and the iterations are on the x-axis. In (b) the same data is presented in whisker plot with a logarithmic scale on the y-axis. 
}
    \label{rex2flog}
\end{figure}
\end{example}

In the following two examples we consider two quadratic matrix polynomials coming from applications. Both these matrix polynomials belong to the NLEVP-collection \cite{BHMS13}. 
\begin{example}
Consider the $5 \times 5$ quadratic matrix polynomial $Q(\lambda) = \lambda^2 M + \lambda D + K$ arising from modelling a two-dimensional three-link mobile manipulator \cite{BHMS13}. 
The $5 \times 5$ coefficient matrices are 
\begin{equation*}
M=\begin{bmatrix}M_0&0\\0&0\\\end{bmatrix}, \ 
D=\begin{bmatrix}D_0&0\\0&0\\\end{bmatrix}, \text{ and }
K=\begin{bmatrix}K_0&-F^T\\F&0\\\end{bmatrix},
\end{equation*}
with
\begin{equation*}
\begin{aligned}
M_0=\begin{bmatrix}18.7532&-7.94493 &7.94494\\-7.94493 &31.8182 &-26.8182\\7.94494 &-26.8182 &26.8182\\\end{bmatrix}, 
& \ 
D_0=\begin{bmatrix}-1.52143& -1.55168& 1.55168\\3.22064 &3.28467 & -3.28467\\-3.22064 &-3.28467 &3.28467\\\end{bmatrix}, \\
K_0=\begin{bmatrix}67.4894 &69.2393 &-69.2393\\69.8124 &1.68624&-1.68617\\-69.8123&-1.68617&-68.2707 \\\end{bmatrix}, 
& \ F_0=\begin{bmatrix}1&0&0\\0&0&1\\\end{bmatrix}.\\
\end{aligned}
\end{equation*}
In Figure \ref{rex3f} we present the decay of the norm of the unstructured part of the perturbation. The changes in the norm of the structured part of the perturbation and in the norms of the transformation matrices are presented in Figures \ref{rex3ff} and \ref{rex3fff}, respectively. 
\begin{figure}[h!]
  \centering
    \begin{subfigure}[b]{0.48\textwidth}    
    \begin{tikzpicture}
  \sbox0{\includegraphics[width=1.0\textwidth]{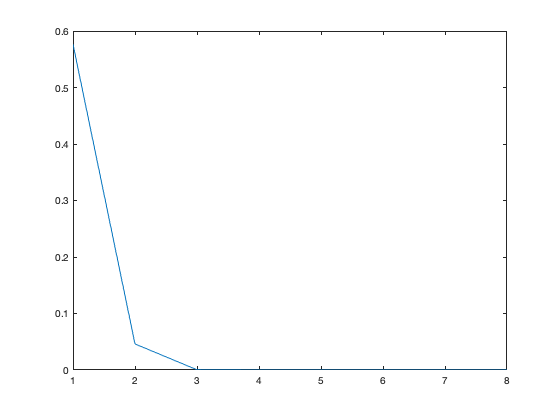}}
  \node[above right,inner sep=0pt] at (0,0)  {\usebox{0}};
  \node[black] at (0.5\wd0,0.05\ht0) {{\tiny\# iterations}};
  \node[black] at (0.02\wd0,0.5\ht0) {{\tiny$||{\cal E}^u||$}};
\end{tikzpicture}
 \caption{}
  \end{subfigure}
  \begin{subfigure}[b]{0.48\textwidth}    
    \begin{tikzpicture}
  \sbox0{\includegraphics[width=1.0\textwidth]{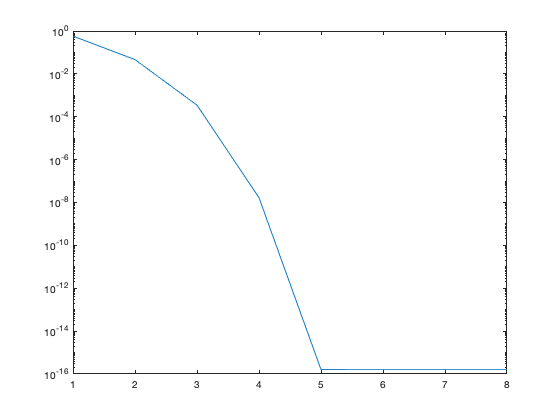}}
  \node[above right,inner sep=0pt] at (0,0)  {\usebox{0}};
  \node[black] at (0.5\wd0,0.05\ht0) {{\tiny\# iterations}};
  \node[black] at (-0.02\wd0,0.5\ht0) {{\tiny$\log ||{\cal E}^u||$}};
\end{tikzpicture}
  \caption{}
  \end{subfigure}
\caption{Changes of the norm of the unstructured part of a perturbation of $5 \times 5$ quadratic matrix polynomial $Q(\lambda) = \lambda^2 M + \lambda D + K$ arising from modelling a two-dimensional three-link mobile manipulator is plotted in (a). The same data but with a logarithmic scale on the y-axis is plotted in (b).}
\label{rex3f}
\end{figure}
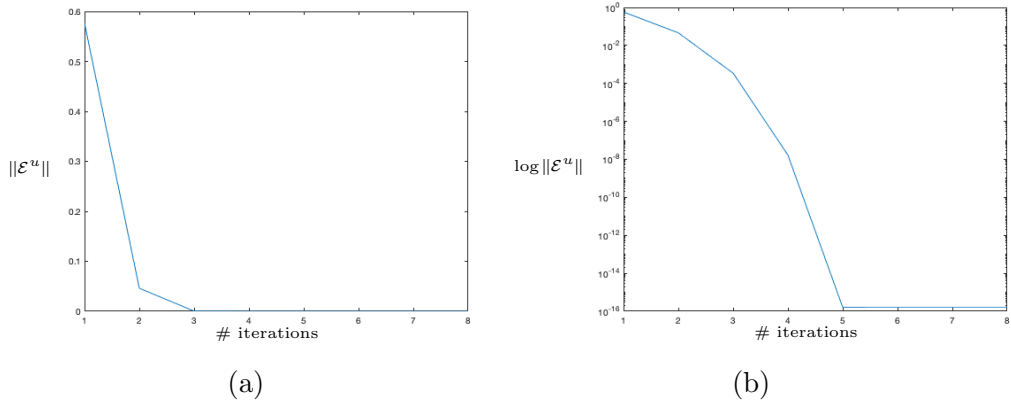
\begin{figure}[h!]
  \centering
\begin{subfigure}[b]{0.48\textwidth}    
   \begin{tikzpicture}
  \sbox0{\includegraphics[width=1\textwidth]{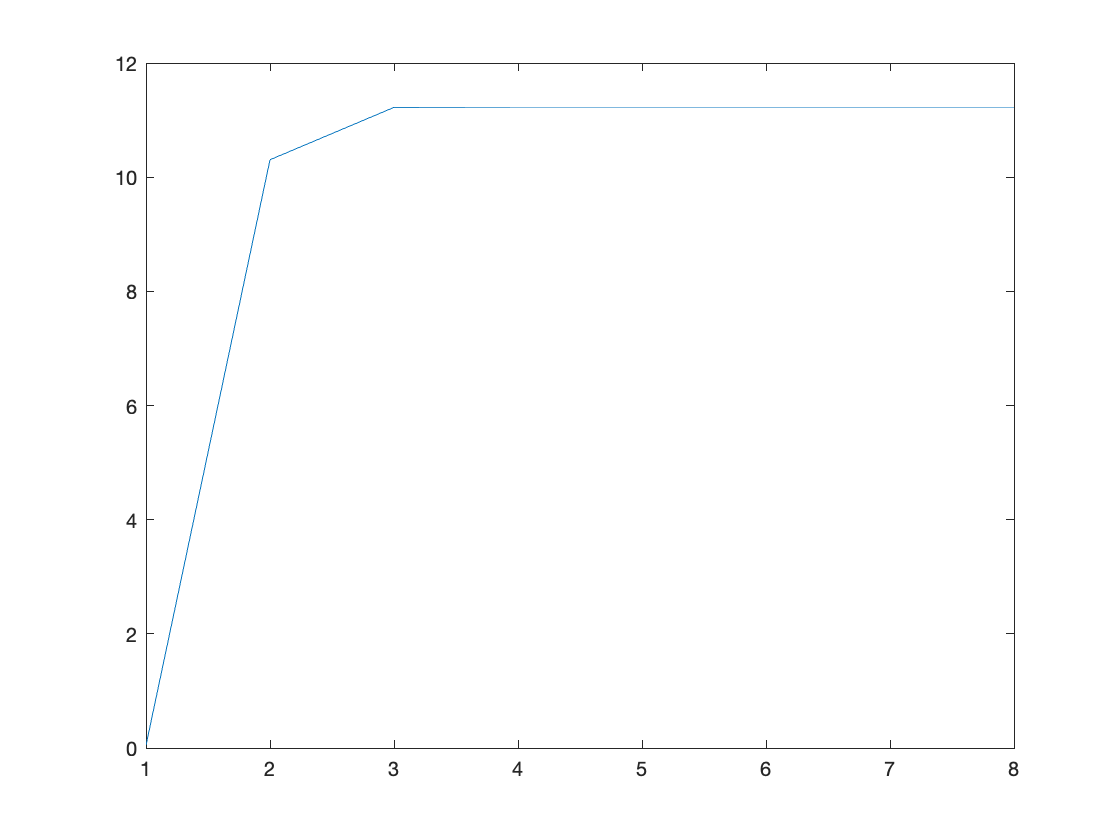}}
  \node[above right,inner sep=0pt] at (0,0)  {\usebox{0}};
  \node[black] at (0.5\wd0,0.05\ht0) {{\tiny\# iterations}};
  \node[black] at (0.02\wd0,0.5\ht0) {{\tiny$||{\cal E}^s||$}};
\end{tikzpicture}
  \caption{}
  \end{subfigure}
  \begin{subfigure}[b]{0.48\textwidth}    
  \begin{tikzpicture}
  \sbox0{\includegraphics[width=1.0\textwidth]{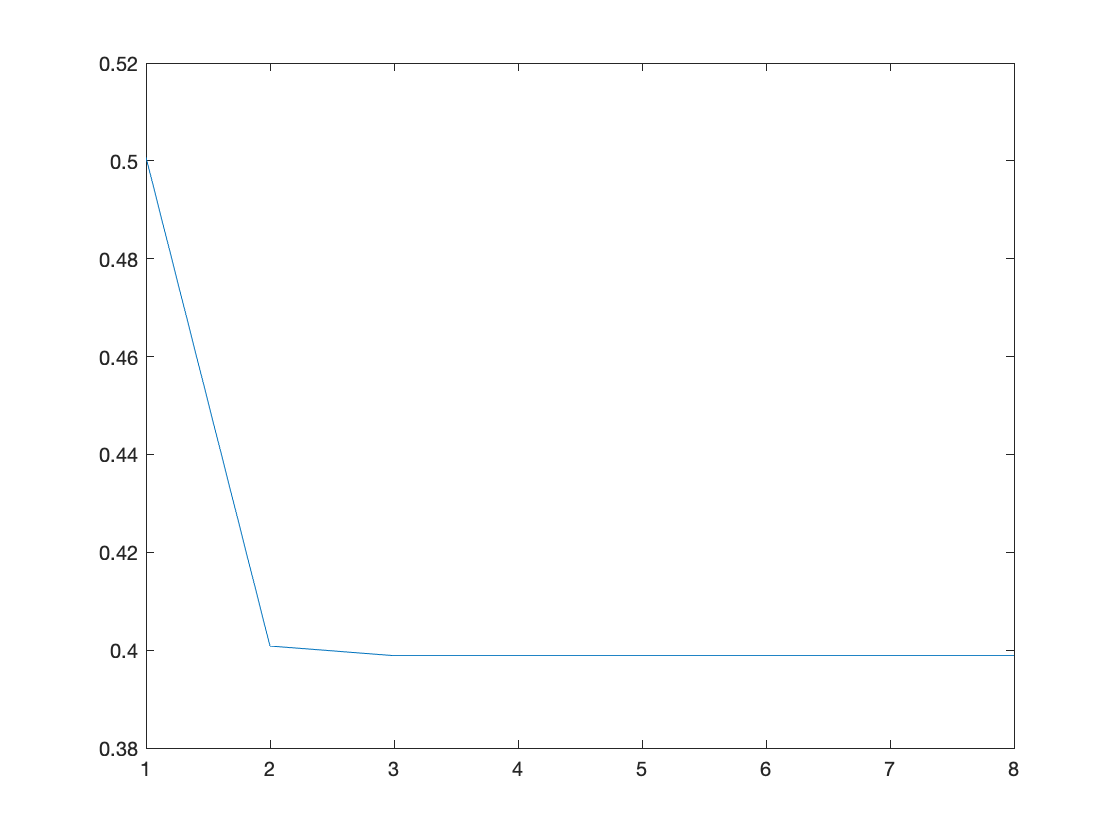}}
  \node[above right,inner sep=0pt] at (0,0)  {\usebox{0}};
  \node[black] at (0.5\wd0,0.05\ht0) {{\tiny\# iterations}};
  \node[black] at (0.02\wd0,0.5\ht0) {{\tiny$||{\cal E}^s||$}};
\end{tikzpicture}
  \caption{}
    \end{subfigure}
\caption{The changes of the norm of the structured part of a perturbation at  each iteration: (a) when we do not normalize the original matrix polynomial; (b) when we normalize the original matrix polynomial.
}
    \label{rex3ff}
\end{figure}
\begin{figure}[h!]
  \centering
  \begin{subfigure}[b]{0.48\textwidth}    
   \begin{tikzpicture}
  \sbox0{\includegraphics[width=1.0\textwidth]{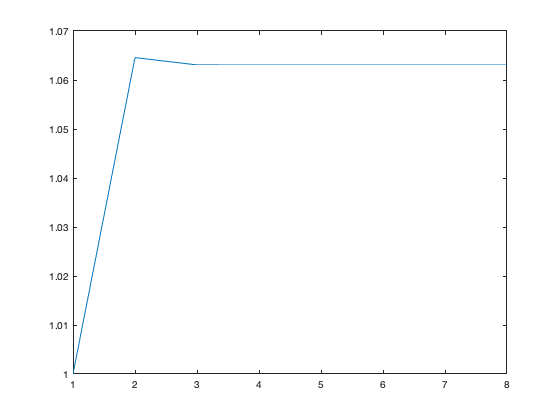}}
  \node[above right,inner sep=0pt] at (0,0)  {\usebox{0}};
  \node[black] at (0.5\wd0,0.05\ht0) {{\tiny\# iterations}};
  \node[black] at (0.02\wd0,0.5\ht0) {{\tiny$||U||_2$}};
\end{tikzpicture}
  \caption{}
  \end{subfigure}
  \begin{subfigure}[b]{0.48\textwidth}    
  \begin{tikzpicture}
  \sbox0{\includegraphics[width=1.0\textwidth]{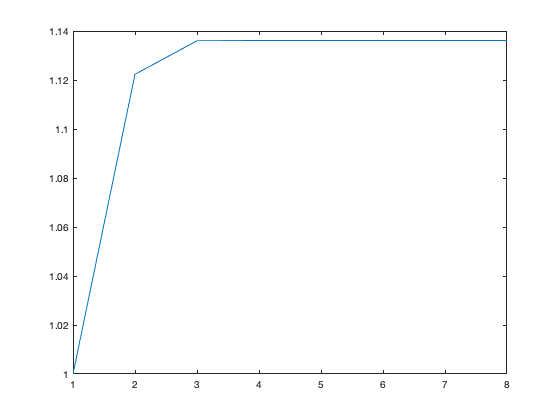}}
  \node[above right,inner sep=0pt] at (0,0)  {\usebox{0}};
  \node[black] at (0.5\wd0,0.05\ht0) {{\tiny\# iterations}};
  \node[black] at (0.02\wd0,0.5\ht0) {{\tiny$||V||_2$}};
\end{tikzpicture}
  \caption{}
  \end{subfigure}
\caption{The changes of the 2-norm of the transformation matrices $U$ and $V$ at  each iteration are plotted in (a) and (b), respectively. Recall that $U \cdot ({\cal C}^1_{P(\lambda)} + {\cal E}_1) \cdot V = {\cal C}^1_{P(\lambda)+ E(\lambda)}$. Note that, $||U||_2$ and $||V||_2$ are close to 1 ($||I||_2=1$).
}
    \label{rex3fff}
\end{figure}
\end{example}

\begin{example}
\label{ex4}
Consider a $21 \times 16$ quadratic matrix
polynomial 
arising from calibration of a surveillance camera
using a human body as a calibration target \cite{BHMS13,MiPa10}. Note that the polynomial is rectangular. 
In Figure \ref{rex4f} we present the decay of the norm of the unstructured part of the perturbation. The changes in the norm of the structured part of the perturbation and in the norms of the transformation matrices are presented in Figures \ref{rex4ff} and \ref{rex4fff}, respectively. 
\begin{figure}[h!]
  \centering
    \begin{subfigure}[b]{0.48\textwidth}    
  \begin{tikzpicture}
  \sbox0{\includegraphics[width=1.0\textwidth]{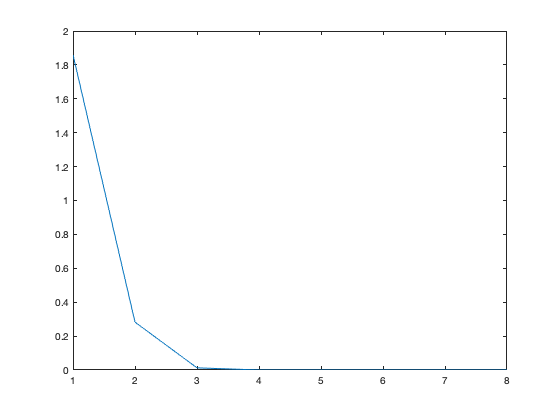}}
  \node[above right,inner sep=0pt] at (0,0)  {\usebox{0}};
  \node[black] at (0.5\wd0,0.05\ht0) {{\tiny\# iterations}};
  \node[black] at (0.02\wd0,0.5\ht0) {{\tiny$||{\cal E}^u||$}};
\end{tikzpicture}
 \caption{}
  \end{subfigure}
  \begin{subfigure}[b]{0.48\textwidth}    
  \begin{tikzpicture}
  \sbox0{\includegraphics[width=1.0\textwidth]{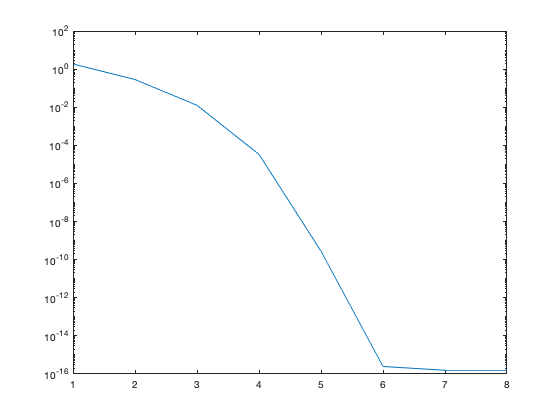}}
  \node[above right,inner sep=0pt] at (0,0)  {\usebox{0}};
  \node[black] at (0.5\wd0,0.05\ht0) {{\tiny\# iterations}};
  \node[black] at (-0.02\wd0,0.5\ht0) {{\tiny$\log ||{\cal E}^u||$}};
\end{tikzpicture}
  \caption{}
  \end{subfigure}
\caption{Changes of the norm of the unstructured part of a perturbation of $21 \times 16$ quadratic matrix polynomial 
arising from calibration of a surveillance camera is plotted in (a). The same convergence data but with a logarithmic scale on the y-axis is plotted in (b).}
\label{rex4f}
\end{figure}
\begin{figure}[h!]
  \centering
  \begin{subfigure}[b]{0.48\textwidth}    
   \begin{tikzpicture}
  \sbox0{\includegraphics[width=1.0\textwidth]{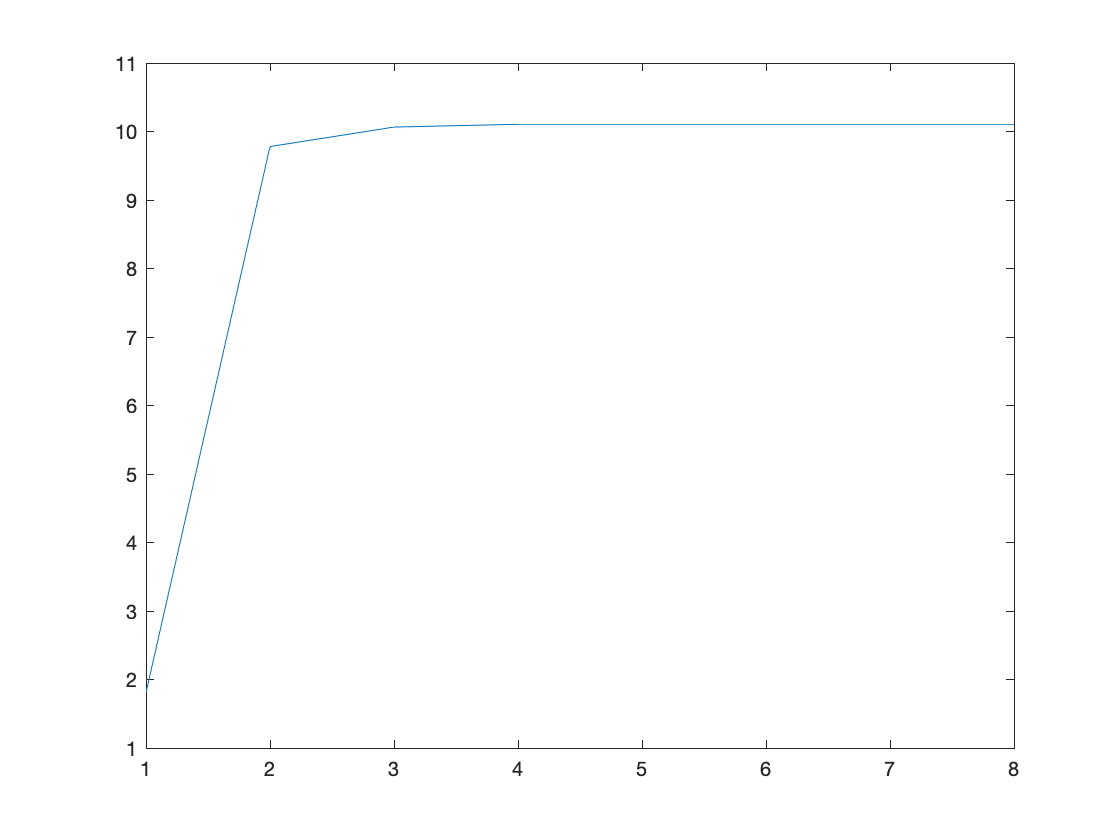}}
  \node[above right,inner sep=0pt] at (0,0)  {\usebox{0}};
  \node[black] at (0.5\wd0,0.05\ht0) {{\tiny\# iterations}};
  \node[black] at (0.02\wd0,0.5\ht0) {{\tiny$||{\cal E}^s||$}};
\end{tikzpicture}
  \caption{}
  \end{subfigure}
  \begin{subfigure}[b]{0.48\textwidth}    
  \begin{tikzpicture}
  \sbox0{\includegraphics[width=1.0\textwidth]{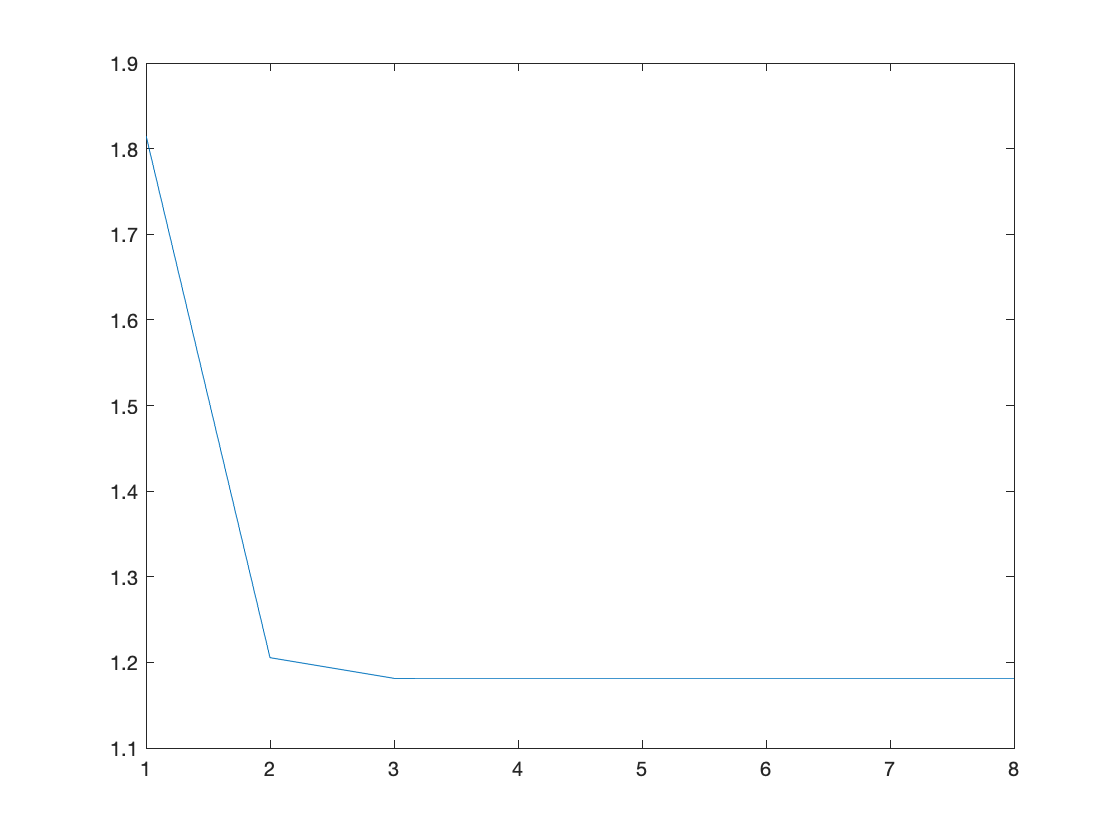}}
  \node[above right,inner sep=0pt] at (0,0)  {\usebox{0}};
  \node[black] at (0.5\wd0,0.05\ht0) {{\tiny\# iterations}};
  \node[black] at (0.02\wd0,0.5\ht0) {{\tiny$||{\cal E}^s||$}};
\end{tikzpicture}
  \caption{}
    \end{subfigure}
\caption{The changes of the norm of the structured part of a perturbation at  each iteration: (a) when we do not normalize the original matrix polynomial; (b) when we normalize the original matrix polynomial.
}
    \label{rex4ff}
\end{figure}
\begin{figure}[h!]
  \centering
  \begin{subfigure}[b]{0.48\textwidth}    
   \begin{tikzpicture}
  \sbox0{\includegraphics[width=1.0\textwidth]{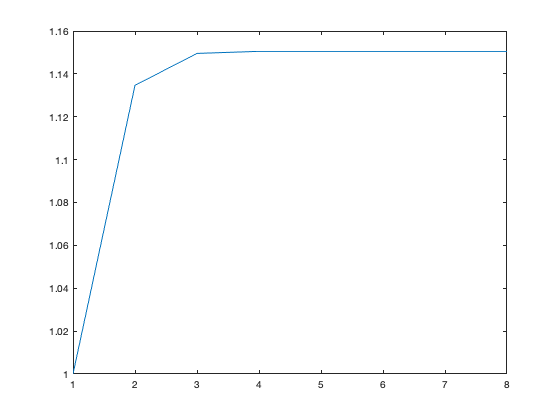}}
  \node[above right,inner sep=0pt] at (0,0)  {\usebox{0}};
  \node[black] at (0.5\wd0,0.05\ht0) {{\tiny\# iterations}};
  \node[black] at (0.02\wd0,0.5\ht0) {{\tiny$||U||_2$}};
\end{tikzpicture}
  \caption{}
  \end{subfigure}
  \begin{subfigure}[b]{0.48\textwidth}    
  \begin{tikzpicture}
  \sbox0{\includegraphics[width=1.0\textwidth]{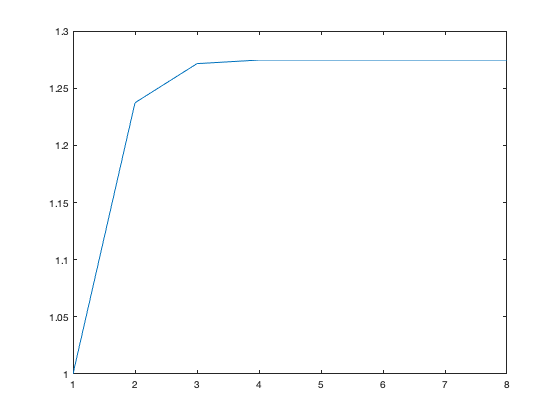}}
  \node[above right,inner sep=0pt] at (0,0)  {\usebox{0}};
  \node[black] at (0.5\wd0,0.05\ht0) {{\tiny\# iterations}};
  \node[black] at (0.02\wd0,0.5\ht0) {{\tiny$||V||_2$}};
\end{tikzpicture}
  \caption{}
  \end{subfigure}
\caption{The changes of the 2-norm of the transformation matrices $U$ and $V$ at  each iteration are plotted in (a) and (b), respectively. Recall that $ U \cdot ({\cal C}^1_{P(\lambda)} + {\cal E}_1) \cdot V = {\cal C}^1_{P(\lambda)+ E(\lambda)}$.}
    \label{rex4fff}
\end{figure}
\end{example}

In the following example we tune the conditioning of the problem and the value of the initial perturbation to test the limits of Algorithm \ref{alg}.

\begin{example}
Consider the $21 \times 16$ quadratic matrix
polynomial from Example \ref{ex4}. We scale the matrix coefficients of this polynomial and increase the initial perturbation to achieve the following goals: (a) making the structured perturbation much larger comparing to the initial perturbation and (b) forcing Algorithm \ref{alg} to diverge. 
Notably, if (a) is achieved, i.e. the limit perturbation that is much larger than the original one, then we may still have the convergence. We summarize the results of our experiment in Table~\ref{table1}. 

\begin{table}
\small
\begin{center}
  \begin{tabular}{ | c c c | c | c | c | c | c | c |}
    \hline
    $\alpha_2$ & $\alpha_1$ & $\alpha_0$  &  $||{\cal E}_1||$  &  $||{\cal E}^s||$ &  $||{\cal E}^s||/||{\cal E}_1||$ &  $||U||_2$  &  $||V||_2$  & conv. \\ \hline
          \multicolumn{9}{|l|}{Entries of ${\cal E}_1$ are equidistributed in $(0,0.001)$:} \\ \hline
                  \footnotesize $1/||Q(\lambda)||$ & \footnotesize  $1/||Q(\lambda)||$& \footnotesize  $1/||Q(\lambda)||$ &0.0083& 0.0044& 0.53 &1.001 &1.002 & yes \\
    1 & 1& 1 &0.0083&  0.24& 28 &1.008 &1.001 & yes \\
    10 & 1& 1 &0.0082& 10.6  & 1295 &1.07 &1.02 & yes \\ \hline
           \multicolumn{9}{|l|}{Entries of ${\cal E}_1$ are equidistributed in $(0,0.01)$:} \\ \hline
                   \footnotesize $1/||Q(\lambda)||$ & \footnotesize  $1/||Q(\lambda)||$& \footnotesize  $1/||Q(\lambda)||$ &0.084& 0.05& 0.6 &1.01 &1.02 & yes \\
    1 & 1& 1 &0.085&  11.5& 135.4 &1.2 &1.08 & yes \\
    10 & 1& 1 &0.083 & 229 & 2752 &1.75 &1.74 & yes \\ \hline
        \multicolumn{9}{|l|}{Entries of ${\cal E}_1$ are equidistributed in $(0,0.1)$:} \\ \hline
        \footnotesize $1/||Q(\lambda)||$ & \footnotesize  $1/||Q(\lambda)||$& \footnotesize  $1/||Q(\lambda)||$ &0.85& 0.33& 0.39 &1.08 &1.13 & yes \\
    1 & 1& 1 &0.84&  45& 54 &1.28 &1.27 & yes \\
    10 & 1& 1 & 0.82 & -- & -- &-- &-- & no \\ \hline
              \multicolumn{9}{|l|}{Entries of ${\cal E}_1$ are equidistributed in $(0,2)$:} \\ \hline
     \footnotesize $1/||Q(\lambda)||$ & \footnotesize  $1/||Q(\lambda)||$& \footnotesize  $1/||Q(\lambda)||$ &17&  -- & -- &-- &-- & no \\
    \hline
  \end{tabular}
  \caption{\label{table1}In the table we show how the choice of the scalars $\alpha_i, i=0,1,2$ in the matrix polynomial $Q(\lambda) = \alpha_2 A_2 \lambda^2 + \alpha_1 A_1 \lambda + \alpha_0 A_0$ and the initial perturbation ${\cal E}_1$ change the norm of the resulting structured perturbation ${\cal E}^s$ and the convergence of the algorithm. 
  }
    \end{center}
\end{table}

\end{example}

\section{Future work}

The method developed in this paper can be directly generalized to the other linearizations, e.g., Fiedler linearizations \cite{AnVo04,DeDM12,DJKV19} or even block-Kronecker linearizations \cite{DLPV18}. Such a generalization may also cover structure-preserving linearizations, see e.g., \cite{Dmyt17}. The existence of structured perturbations for these broader classes of linearizations follows, e.g., from \cite{Dmyt17,DJKV19, DLPV18}. Such a generalization will require solving the corresponding structured coupled Sylvester equations, or at least the corresponding structured least-squares problem.   
%
%

\section*{Acknowledgements}
The author is thankful to Zhaojun Bai and Froil\'an Dopico for the useful discussions on this paper.


{\tiny
\bibliographystyle{abbrv}
\bibliography{library}

\begin{thebibliography}{10}

\bibitem{AnVo04}
E.~Antoniou and S.~Vologiannidis.
\newblock {A} new family of companion forms of polynomial matrices.
\newblock {\em {E}lectron. {J}. {L}inear {A}lgebra}, 11:78--87, 2004.

\bibitem{AvDT13}
H.~Avron, A.~Druinsky, and S.~Toledo.
\newblock Spectral condition-number estimation of large sparse matrices.
\newblock {\em arXiv preprint arXiv:1301.1107}, 2013.

\bibitem{BHMS13}
T.~Betcke, N.~Higham, V.~Mehrmann, C.~Schr\"{o}der, and F.~Tisseur.
\newblock {NLEVP}: {A} {C}ollection of {N}onlinear {E}igenvalue {P}roblems.
\newblock {\em {ACM} {T}rans. {M}ath. {S}oftware}, 39(2):7:1--7:28, 2013.

\bibitem{ByHM98}
R.~Byers, C.~He, and V.~Mehrmann.
\newblock Where is the nearest non-regular pencil?
\newblock {\em {L}inear {A}lgebra {A}ppl.}, 285(1):81 -- 105, 1998.

\bibitem{ChRa08}
J.-P. Chehab and M.~Raydan.
\newblock Geometrical properties of the frobenius condition number for positive
  definite matrices.
\newblock {\em {L}inear {A}lgebra {A}ppl.}, 429(8):2089 -- 2097, 2008.

\bibitem{DeDM12}
F.~{De~Ter\'{a}n}, F.~M. Dopico, and D.~S. Mackey.
\newblock {F}iedler companion linearizations for rectangular matrix
  polynomials.
\newblock {\em {L}inear {A}lgebra {A}ppl.}, 437(3):957--991, 2012.

\bibitem{Dmyt16}
A.~Dmytryshyn.
\newblock {M}iniversal deformations of pairs of skew-symmetric matrices under
  congruence.
\newblock {\em {L}inear {A}lgebra {A}ppl.}, 506:506--534, 2016.

\bibitem{Dmyt17}
A.~Dmytryshyn.
\newblock {S}tructure preserving stratification of skew-symmetric matrix
  polynomials.
\newblock {\em {L}inear {A}lgebra {A}ppl.}, 532:266--286, 2017.

\bibitem{Dmyt19}
A.~Dmytryshyn.
\newblock Miniversal deformations of pairs of symmetric matrices under
  congruence.
\newblock {\em {L}inear {A}lgebra {A}ppl.}, 568:84 --105, 2019.

\bibitem{DmDo17}
A.~Dmytryshyn and F.~M. Dopico.
\newblock {G}eneric matrix polynomials with fixed rank and fixed degree.
\newblock {\em {L}inear {A}lgebra {A}ppl.}, 535:213--230, 2017.

\bibitem{DFKK15}
A.~Dmytryshyn, V.~Futorny, B.~K{\aa}gstr\"{o}m, L.~Klimenko, and V.~Sergeichuk.
\newblock {C}hange of the congruence canonical form of 2-by-2 and 3-by-3
  matrices under perturbations and bundles of matrices under congruence.
\newblock {\em {L}inear {A}lgebra {A}ppl.}, 469:305--334, 2015.

\bibitem{DmFS12}
A.~Dmytryshyn, V.~Futorny, and V.~Sergeichuk.
\newblock {M}iniversal deformations of matrices of bilinear forms.
\newblock {\em {L}inear {A}lgebra {A}ppl.}, 436:2670--2700, 2012.

\bibitem{DmFS14}
A.~Dmytryshyn, V.~Futorny, and V.~Sergeichuk.
\newblock {M}iniversal deformations of matrices under *congruence and reducing
  transformations.
\newblock {\em {L}inear {A}lgebra {A}ppl.}, 446:388--420, 2014.

\bibitem{DJKV19}
A.~Dmytryshyn, S.~Johansson, B.~K{\aa}gstr{\"{o}}m, and P.~Van~Dooren.
\newblock {G}eometry of matrix polynomial spaces.
\newblock {\em Found. Comput. Math.}, (20):423--450, 2020.

\bibitem{DLPV18}
F.~M. Dopico, P.~Lawrence, J.~P\'erez, and P.~Van~Dooren.
\newblock Block {K}ronecker linearizations of matrix polynomials and their
  backward errors.
\newblock {\em Numer. Math.}, (140):373--426, 2018.

\bibitem{FuKS14}
V.~Futorny, V.~Klimenko, and V.~Sergeichuk.
\newblock {C}hange of the *congruence canonical form of 2-by-2~matrices under
  perturbations.
\newblock {\em {E}lectron. {J}. {L}inear {A}lgebra}, 27, 2014.

\bibitem{FKKS19}
V.~Futorny, T.~Klymchuk, V.~V. Sergeichuk, and N.~Shvai.
\newblock A constructive proof of pokrzywa's theorem about perturbations of
  matrix pencils, 2019.

\bibitem{GiHL17}
M.~Giesbrecht, J.~Haraldson, and G.~Labahn.
\newblock Computing the nearest rank-deficient matrix polynomial.
\newblock In {\em Proceedings of the 2017 ACM on International Symposium on
  Symbolic and Algebraic Computation}, ISSAC '17, pages 181--188, New York, NY,
  USA, 2017. ACM.

\bibitem{GuLM17}
N.~Guglielmi, C.~Lubich, and V.~Mehrmann.
\newblock On the nearest singular matrix pencil.
\newblock {\em SIAM J. Matrix Analysis Applications}, 38:776--806, 2017.

\bibitem{GuTi17}
S.~Güttel and F.~Tisseur.
\newblock The nonlinear eigenvalue problem.
\newblock {\em Acta Numerica}, 26:1–94, 2017.

\bibitem{HiMM04}
A.~Hilliges, C.~Mehl, and V.~Mehrmann.
\newblock {O}n the solution of palindromic eigenvalue problems.
\newblock In {\em {P}roceedings of the 4th {E}uropean {C}ongress on
  {C}omputational {M}ethods in {A}pplied {S}ciences and {E}ngineering
  ({ECCOMAS}). {J}yv{\"a}skyl{\"a}, {F}inland}, 2004.

\bibitem{JoKV13}
S.~Johansson, B.~K{\aa}gstr\"{o}m, and P.~{Van~Dooren}.
\newblock {S}tratification of full rank polynomial matrices.
\newblock {\em {L}inear {A}lgebra {A}ppl.}, 439:1062--1090, 2013.

\bibitem{KaTi15}
L.~Karlsson and F.~Tisseur.
\newblock {A}lgorithms for {H}essenberg-{T}riangular {R}eduction of {F}iedler
  {L}inearization of {M}atrix {P}olynomials.
\newblock {\em {SIAM} {J}ournal on {S}cientific {C}omputing}, 37(3):C384--C414,
  2015.

\bibitem{KrSW09}
D.~Kressner, C.~Schr{\"o}der, and D.~Watkins.
\newblock {I}mplicit {QR} algorithms for palindromic and even eigenvalue
  problems.
\newblock {\em {N}umerical {A}lgorithms}, 51(2):209--238, 2009.

\bibitem{MaMT15}
D.~S. Mackey, N.~Mackey, and F.~Tisseur.
\newblock {P}olynomial {E}igenvalue {P}roblems: {T}heory, {C}omputation, and
  {S}tructure.
\newblock In {\em {N}umerical {A}lgebra, {M}atrix {T}heory,
  {D}ifferential-{A}lgebraic {E}quations and {C}ontrol {T}heory}, pages
  319--348. Springer, 2015.

\bibitem{MiPa10}
B.~Micus{\'i}k and T.~Pajdla.
\newblock Simultaneous surveillance camera calibration and foot-head homology
  estimation from human detections.
\newblock {\em 2010 IEEE Computer Society Conference on Computer Vision and
  Pattern Recognition}, pages 1562--1569, 2010.

\bibitem{SuGo13}
A.~Suárez and L.~González.
\newblock Normalized frobenius condition number of the orthogonal projections
  of the identity.
\newblock {\em J. Math. Anal. Appl.}, 400(2):510 -- 516, 2013.

\bibitem{Tiss00}
F.~Tisseur.
\newblock Backward error and condition of polynomial eigenvalue problems.
\newblock {\em {L}inear {A}lgebra {A}ppl.}, 309(1):339--361, 2000.

\bibitem{TiMe01}
F.~Tisseur and K.~Meerbergen.
\newblock {T}he quadratic eigenvalue problem.
\newblock {\em {SIAM} {R}eview}, 43(2):235--286, 2001.

\bibitem{Tren99}
W.~F. Trench.
\newblock Invertibly convergent infinite products of matrices.
\newblock {\em J. Comput. Appl. Math.}, 101(1):255--263, 1999.

\bibitem{VaDe83}
P.~{Van~Dooren} and P.~Dewilde.
\newblock {T}he eigenstructure of a polynomial matrix: Computational aspects.
\newblock {\em {L}inear {A}lgebra {A}ppl.}, 50:545--579, 1983.

\end{thebibliography}
}
\end{document}